\documentclass{article}
\usepackage{comment}
\usepackage{graphicx} 
\usepackage{amssymb,amsmath,amsthm}
\newtheorem{theorem}{Theorem}
\newtheorem{lemma}[theorem]{Lemma}
\newtheorem{corollary}[theorem]{Corollary}
\newtheorem{proposition}[theorem]{Proposition}
\theoremstyle{definition}
\usepackage{geometry}

\usepackage{mathtools}
\usepackage{hyperref}
\usepackage{multicol}

\title{Sharp thresholds for higher powers of Hamilton cycles in random graphs}
\author{Tam\'as Makai \thanks{Department of Mathematics, LMU Munich, Munich 80333 Germany. Email: \href{mailto:makai@math.lmu.de}{\nolinkurl{makai@math.lmu.de}.}}
   \and Matija Pasch\thanks{Email: \href{mailto:matija.pasch@gmail.com}{\nolinkurl{matija.pasch@gmail.com}}}
   \and Kalina Petrova\thanks{Institute of Science and Technology Austria (ISTA), Klosterneurburg 3400, Austria. Email: \href{mailto:kalina.petrova@ist.ac.at}{\nolinkurl{kalina.petrova@ist.ac.at}}.
    	This project has received funding from the European Union’s Horizon 2020 research and innovation programme under the Marie Skłodowska-Curie grant agreement No 101034413.
        $\begin{array}{l}
        \hspace{-0.2cm}
        \includegraphics[width=5.5mm, height=4mm]{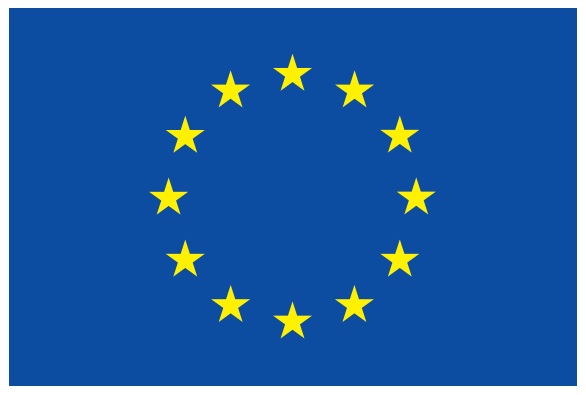}
        \end{array}
        $
        }
   \and Leon Schiller \thanks{Hasso Plattner Institute, University of Potsdam, Potsdam 14482, Germany. Email: \href{mailto:leon.schiller@hpi.de}{\nolinkurl{leon.schiller@hpi.de}} }}
\date{}

\DeclareMathOperator{\aut}{aut}

\usepackage[dvipsnames]{xcolor}

\begin{document}

\maketitle

\begin{abstract}
For $k \geq 4$, we establish that $p = (e/n)^{1/k}$ is a sharp threshold for the existence of the $k$-th power $H$ of a Hamilton cycle in the binomial random graph model. Our proof builds upon an approach by Riordan based on the second moment method, which previously established a weak threshold for $H$. This method expresses the second moment bound through contributions of subgraphs of $H$, with two key quantities: the number of copies of each subgraph in $H$ and the subgraphs' densities. We control these two quantities more precisely by carefully restructuring Riordan's proof and treating sparse and dense subgraphs of $H$ separately. This allows us to determine the exact constant in the threshold.
\end{abstract}
\section{Introduction}
Research devoted to problems concerning Hamilton cycles in the binomial random graph model $G(n,p)$\footnote{In the \emph{binomial random graph} $G(n,p)$ on $n$ vertices, each potential edge is included with probability $p=p(n)$ independently.} began with the introduction of random graphs by Erd\H{o}s and R\'enyi~\cite{erdos1960} in 1960.
The extensive history of inquiry into this topic is well exemplified by Frieze's survey~\cite{frieze2023}.
First, the threshold for the appearance of the Hamilton cycle has been established as $p=(1+o(1))\ln{n}/n$ and in later works the hitting time and other more precise results were derived \cite{ajtai1985,bollobas1984,erdos1960,komlos1983,korsunov1976,posa1976}.

K\"uhn and Osthus~\cite{kuhn2012} were the first to investigate the threshold for a $k$-th power of a Hamilton cycle.\footnote{The \emph{$k$-th power} of a graph $G$ is obtained by including an edge $\{u,v\}$ if the distance between $u$ and $v$ in $G$ is at most $k$.} They proved the upper bound $n^{-1/2+\varepsilon}$ for the case $k=2$. Since then, the threshold for the square of a Hamilton cycle has been the focus of a great deal of research. Nenadov and \v{S}kori\'c~\cite{nenadov2019} established the upper bound $O(\ln^4{n}/n^{1/2})$, which was subsequently improved to $O(\ln^3{n}/n^{1/2})$ by Fischer, \v{S}kori\'c, Steger, and Truji\'c~\cite{fischer2018}, and then to $O(\ln^2{n}/n^{1/2})$ by Montgomery~\cite{montgomery2018}, all employing the absorption method and connection techniques. The correct order of magnitude $\Theta(n^{-1/2})$ was finally shown by Kahn, Narayanan, and Park~\cite{kahn2021}, who follow a different approach, based on the resolution of the fractional `expectation threshold' conjecture~\cite{frankston2019}. They conjectured that the sharp threshold should be $\sqrt{e/n}$ --- a hypothesis that has attracted significant interest in the community~\cite{frieze2023,perkins2024}.

Turning to higher powers of Hamilton cycles, K\"uhn and Osthus~\cite{kuhn2012} observed that for $k \geq 3$, the correct order of magnitude $\Theta(n^{-1/k})$ of the threshold for the $k$-th power follows from the first moment bound and a result of Riordan~\cite{riordan2000}, which applies to a broad class of spanning subgraphs. While~\cite{dudek2020} suggests that Riordan's theorem also implies the sharp threshold, this is not directly the case, as applying~\cite{riordan2000} requires $pn^{1/k} \rightarrow \infty$. In this paper, we extend Riordan's approach to show that the sharp threshold for the $k$-th power of a Hamilton cycle for $k\geq 4$ matches the lower bound given by the first moment method.

\begin{theorem}\label{thm:main}
For all $k\geq 4$, we have that $p^*=(e/n)^{1/k}$ is a sharp threshold for the existence of a $k$-th power $H$ of a Hamilton cycle in $G(n,p)$. That is, for all $\varepsilon>0$ and all $p\le(1-\varepsilon)p^*$, there is whp\footnote{A sequence $E_n$ of events holds \emph{with high probability} (whp), if the probability of $E_n$ tends to one as $n$ tends to infinity.} no copy of $H$ in $G(n,p)$, and for all $p\ge(1+\varepsilon)p^*$, there is whp a copy of $H$ in $G(n,p)$.
\end{theorem}
In the withdrawn draft~\cite{zhukovskii2023}, Zhukovskii presented a more general result that established a sharp threshold for a wider class of spanning regular subgraphs. He used a different method, namely fragmentation, which was also employed by Kahn, Narayanan, and Park in \cite{kahn2021}. As we learned via private communication, the draft was withdrawn due to a mistake in the second moment calculation which, in particular, led to an incorrect argument for the second power of a Hamilton cycle. Zhukovskii is planning to upload a corrected version of this result.

\subsection{Proof strategy}
Theorem~\ref{thm:main} consists of two statements, the lower bound and the upper bound on the threshold for the containment of the $k$-th power of a Hamilton cycle in $G(n,p)$. The lower bound follows easily from the first moment method and has been known for a long time. For completeness, we provide it in the appendix (Section~\ref{subsec:proof_of_main_result}). In the remainder of the paper, we focus on the upper bound.

For the upper bound, our approach is inspired by the method in~\cite{riordan2000}, where Riordan gives sufficient conditions for the appearance of a large family of spanning subgraphs $H$ in $G(n,p)$. The author makes use of the second moment method. As a consequence of that, the main objective becomes showing that the variance of the number $X$ of copies of $H$ in $G(n,p)$ is negligible relative to its squared expectation. The key idea then is to write the ratio between the variance of $X$ and $(\mathbb{E}(X))^2$ as a sum of contributions of all subgraphs of $H$. This makes it much easier to estimate each term.

The analysis in~\cite{riordan2000} is not tight enough to yield the sharp threshold for powers of Hamilton cycles. 
The first additional key idea we introduce is to treat sparse and dense subgraphs of $H$ separately. 
In particular, the bound from~\cite{riordan2000} on the contributions of sparse subgraphs is sufficient. 
On the other hand, we significantly improve the bound on the contributions of dense subgraphs, which we can do since we focus on a specific spanning graph, as opposed to Riordan's general result. 
For this, we rely on the fact that when $F \subset H$ has almost the same density as $H$, then there are very few copies of $F$ in $H$, since we do not have much freedom at each step when embedding $F$ into $H$. 

The final step is to sum up the bounds over every subgraph of $H$.
Since we differentiate between dense and sparse subgraphs of $H$, it is not possible to factorize the sum of the bounds into components' contributions (and then sum the bounds for connected subgraphs over the same vertex set) as in~\cite{riordan2000} (for more details, see Section~\ref{sec:proof_strategy}). Instead, we first establish the partial sums of the bounds over every subgraph with the same vertex set and component structure. We upper bound these partial sums in a different way, depending on the density of the densest graph with a fixed component structure. Summing these contributions gives a suitable upper bound on the variance of the number of copies of $H$, which together with the second moment method implies our result.

\section{Preliminaries and proof outline}
We start by introducing the relevant notions and notation in Section~\ref{sec:notions_notation}. Next, we give a proof of our main results modulo the lemmas in Section~\ref{sec:proof_strategy}.
Then we give an overview of the rest of the paper in Section~\ref{sec:overview}.
\subsection{Notions and notation}\label{sec:notions_notation}
Throughout, we consider graphs with vertices in $[n]=\{1,\dots,n\}$ for $n$ sufficiently large. 
Let $K_n$ be the complete graph on the vertex set $[n]$. The uniform Erd\H{o}s–R\'enyi model $G(n,M)$ is the random graph on the vertex set $[n]$ in which $M$ of the $\binom{n}{2}$ potential edges are sampled uniformly at random. We use the standard Bachmann-Landau notation $o(\cdot)$, $O(\cdot)$, $\Theta(\cdot)$, $\Omega(\cdot)$ and $\omega(\cdot)$, exclusively with respect to $n$. We omit rounding of real numbers to their nearest integers whenever it is not essential for the argument.

For a graph $F$, let $V(F)$ be the set of vertices of $F$, $E(F)$ --- the set of edges of $F$, and $C(F)$ --- the set of connected components of $F$.
Let $|F|=|V(F)|$ be the order of $F$, $c(F)=|C(F)|$ --- the number of components of $F$,\footnote{This is denoted by $k(F)$ in \cite{riordan2000}, but we reserve $k$ for the power of the cycle.} and $e(F)=|E(F)|$ --- the number of edges of $F$. For graphs $F$ and $G$, we denote by $F \subset G$ the statement that $F$ is a (not necessarily proper) subgraph of $G$, and similarly for a subset $T$ of a set $S$. We denote by $H[S]$ the induced subgraph of $H$ on the vertex set $S \subset V(H)$.

Since our proof decomposes various subgraphs of $H$ into different types of connected components, we need some specific definitions for that purpose. Let $r(F)=|F|-c(F)$ be the rank of $F$. We call an edge \emph{isolated} if it is the only edge in its component.
Let
$$F^\circ=\bigcup_{C\in C(F):|C|>2}C$$
be $F$ with all isolated vertices and isolated edges discarded. If $e(F^\circ) >0$, we define the density of $F$ as
$$\gamma(F)=\frac{e(F^\circ)}{|F^\circ|-2c(F^\circ)}=\frac{e(F^\circ)}{r(F^\circ)-c(F^\circ)}.$$
Note that $\gamma(F)=\sum_{C\in C(F^\circ)}\mu_F(C)\gamma(C)$ can be seen as an expected value of $\gamma(C)$ with respect to the probability distribution given by
\[
\mu_F(C)=\frac{|C|-2}{\sum_{C\in C(F^\circ)}(|C|-2)}
\]
for each $C\in C(F^\circ)$.

A graph $F$ is \emph{good} if $V(F)=[n]$ and for all $C\in C(F)$, we have $|C|\neq 2$, i.e.~$F$ has no isolated edges. Notice that for good $F$, we have $e(F)=e(F^\circ)$ and $r(F)=r(F^\circ)$.
For graphs $F,G$, let $X_F(G)$ be the number of copies of $F$ in $G$.

For $s>0$, we sometimes use cyclic addition on $[s]$, e.g.~$s-1+3=2$, and the corresponding cyclic order. 

Throughout the paper, we fix $k \ge 4$ and we denote by $H$ the $k$-th power of the cycle with vertices $[n]$ and edges $\{v,v+1\}$ for $v\in[n]$. That is, the edges of $H$ are given by $\{v,v+i\}$ for $v\in[n]$ and $i \in [k]$.
For $3\le s\le n$, let
$$\gamma(s)=\max_{F\subset H:|F|=s, e(F^\circ)>0}\gamma(F)$$ 
be the maximal density of any subgraph $F$ of $H$ on $s$ vertices with $e(F^\circ) > 0$. Furthermore, let $H_s\subset H$ be the subgraph induced by $[s]$ and let $P_s$ be the $k$-th power of the path with vertices $[s]$ and edges $\{v,v+1\}$ for $v\in[s-1]$.

Let $F\subset H$ be a subgraph of $H$. We define the \emph{completion $F^*\subset H$} of $F$ as the unique graph given by $$F^*=\bigcup_{C \in C(F)} H\big[V(C)\big],$$ that is, the union of all induced subgraphs of $H$ with vertex sets $V(C)$ for $C \in C(F)$.

We fix some function $\zeta= \zeta(n)$ such that $\zeta=o(1)$ and $\zeta=\omega(1/\ln(n))$. For any parameter $\zeta' > 0$, we say that $F\subset H$ is \emph{$\zeta'$-dense} if $e(F^\circ)>0$ and $\gamma(F)\ge\gamma(n)-\zeta'$. Conversely, we refer to $F \subset H$ as \emph{$\zeta'$-sparse} if $e(F)>0$, and either $e(F^\circ) = 0$, or $\gamma(F)< \gamma(n)-\zeta'$. Note that the empty graph is the only graph that is neither $\zeta'$-dense nor $\zeta'$-sparse. If a graph is $\zeta$-dense (respectively $\zeta$-sparse), we call it simply dense (respectively sparse). Finally, we say that $F \subset H$ is \emph{completion-dense} if $e(F^\circ) > 0$ and $\gamma(F) \geq (1- \zeta) \gamma(F^*)$, and \emph{completion-sparse} if $e(F)>0$, and either $e(F^\circ)=0$, or $\gamma(F) < (1- \zeta) \gamma(F^*)$. Note that completion-density is defined relative to the completion of the considered graph.

Let $N=e(K_n)=\binom n2$.
Furthermore, let $\Delta=2k$ and notice that $H$ is $\Delta$-regular.
We have $e(H)=kn$. Let $\alpha=kn/N$, meaning that $e(H) = \alpha N$.
\subsection{Proof of the main result}\label{sec:proof_strategy}

The proof of the upper bound in Theorem~\ref{thm:main} partially follows the proof strategy in~\cite{riordan2000}. In particular, we also work in $G(n,M)$ instead of $G(n,p)$. As noted in the remark on page 131 in \cite{riordan2000}, this is necessary since the variance of the number of copies of $H$ in $G(n,p)$ is too high for the second moment method to yield the desired result. Indeed, when considering the probability that two overlapping copies of $H$ both appear, one can observe that the appearance of one copy significantly increases the probability that the other copy is also present. Another way to think of this phenomenon is that planting $H$ in $G(n,p)$ significantly increases the expected number of edges, whereas in the uniform model $G(n,M)$ that is not the case. We now state our main result for $G(n,M)$.

\begin{theorem}\label{thm:main_uniform}
For all $k \geq 4$ and $\varepsilon >0$, the following holds for the sequence $p^*=(e/n)^{1/k}$. For all $p\ge(1+\varepsilon)p^*$, the uniform Erd\H{o}s-R\'enyi random graph $G(n,pN)$ whp contains a copy of the $k$-th power $H$ of a Hamilton cycle.
\end{theorem}
Note that the upper bound in Theorem~\ref{thm:main} follows from Theorem~\ref{thm:main_uniform} using the concentration of the number of edges. For completeness, we provide the details in the appendix (Section~\ref{subsec:proof_of_main_result}).

One of the main results in~\cite{riordan2000}, namely Theorem 2.1, states that for every $n$-vertex graph $\bar{H}$ with maximum degree $\bar{\Delta} = \bar{\Delta}(\bar{H})$, $N = \binom{n}{2}$, $\bar{\alpha} = \bar{\alpha}(\bar{H}) = e(\bar{H})/N$, and 
\begin{align}
    \label{eq:gammabar}
    \bar{\gamma} = \bar{\gamma}(\bar{H}) = \max_{F \subset \bar{H}: |F| \geq 3} \Bigg\{\frac{e(F)}{|F| - 2}\Bigg\},
\end{align}
the conditions $\bar{\alpha} N \geq n$, $pN = \omega(1)$, $(1-p)\sqrt{n} = \omega(1)$, $np^{\bar{\gamma}}/\bar{\Delta}^4 = \omega(1)$ are sufficient to conclude that $G(n, pN)$ whp contains a copy of $\bar{H}$. One can check that for $\bar{H} = H$ and $p = (1+\varepsilon)(e/n)^{1/k}$, all these conditions except the last one hold. In particular, $np^{\bar{\gamma}(H)}/\bar{\Delta}^4 = O(1)$ whenever $p = (1+\varepsilon)(e/n)^{1/k}$ since $\bar{\gamma}(H) \geq \frac{nk}{n-2}$. Thus, Riordan's theorem does not directly yield a sharp threshold for powers of Hamilton cycles, and we need to develop and refine his argument to obtain our result.

From this point on, we fix $\varepsilon>0$ and $p=(1+\varepsilon)(e/n)^{1/k}$. We start by applying the second moment method, as in~\cite{riordan2000}. Letting $X$ denote the number of copies of some spanning graph $\bar{H}$ in $G(n,\bar{p}N)$, we have by Chebyshev's inequality that
\begin{align}
\label{eq:secondMoment}
\mathbb{P}(X=0) \leq \mathbb{P}\left((X-\mu)^2 \geq \mu^2\right) \leq \frac{\mathrm{Var}(X)}{\mu^2},
\end{align}
where $\mu = \mathbb{E}(X)$. The proof in~\cite{riordan2000} defines $f = 1 + \frac{\mathrm{Var}(X)}{\mu^2} $, where $f = f(\bar{H}, \bar{p})$ depends on $\bar{p}$ and $\bar{H}$ (see~\cite[Lemma 3.1]{riordan2000}). It is then enough to show that $f = 1+ o(1),$ which would imply $\mathbb{P}(X=0) = o(1)$. We follow that strategy, starting by stating an upper bound on $f$ given in~\cite{riordan2000} which applies to our setting. For this purpose, we define the following quantity. Note that the sum below considers only spanning subgraphs of $\bar{H}$, but these may have isolated vertices.
$$ S_{\bar{H},\bar{p}}=\sum_{\substack{F\subset \bar{H},\\  |F|=n}}\left(\frac{1}{\bar{p}}-1\right)^{e(F)}\left(1+n^{-1/2}\right)^{r(F)}\frac{X_F(\bar{H})}{X_F(K_n)}.$$
A key idea in~\cite{riordan2000} is to rewrite $f$ as a sum of contributions over all spanning subgraphs of $\bar{H}$. This then facilitates the following upper bound on $f$ in terms of $S_{\bar{H},\bar{p}}$.
\begin{lemma}[Lemma 4.2 in~\cite{riordan2000}]\label{lem:lemma4.2}
Suppose that for some $\bar{p}$ and some $n$-vertex graph $\bar{H}$ with maximum degree $\bar{\Delta} = \bar{\Delta}(\bar{H})$, $N = \binom{n}{2}$, and $\bar{\alpha} = \bar{\alpha}(\bar{H}) = e(\bar{H})/N$, the following conditions hold 
\begin{multicols}{3}
\begin{itemize}
\item $\bar{\alpha} N \geq n$;
\item $\bar{p}N = \omega(1)$;
\item $(1-\bar{p})\sqrt{n} = \omega(1)$;
\item $n\bar{p}^2 / \bar{\Delta}^4 = \omega(1)$;
\item $\bar{\alpha}^3 N/ \bar{p}^{2} = o(1)$.
\end{itemize}
\end{multicols}
Then
$$ f(\bar{H},\bar{p}) \leq \big(1+o(1)\big) \exp\left(- \frac{1-\bar{p}}{\bar{p}}\bar{\alpha}^2N\right) S_{\bar{H}, \bar{p}}.$$
\end{lemma}
One can check that the conditions of Lemma~\ref{lem:lemma4.2} are satisfied for $\bar{H} = H$, $\bar{p} = p$, $\bar{\alpha} = \alpha$, $\bar{\Delta} = \Delta$. Thus, we get that $f(H,p) \leq \big(1+o(1)\big) e^{- \frac{1-p}{p}\alpha^2N} S_{H,p}$. From now on, we focus on upper bounding $S_{H,p}$, which we simply denote by $S_H$. We also fix $f=f(H,p)$.

Next, the proof in~\cite{riordan2000} expresses $S_H$ in terms of a sum $S'_H$ of the contributions of all good subgraphs, i.e., graphs with no isolated edges, each multiplied by a factor that accounts for adding all the possible sets of isolated edges to it. Here we make our first modification to~\cite{riordan2000}, which is that we define $S'_H$ so that it gives a tighter upper bound on $S_H$. With $\sum^{\prime}$ denoting the summation only over good graphs, let
\begin{align}\label{eq:S'def} \hspace{-0.5em}
S'_H = \sum_{\substack{F\subset H, \\|F| = n}}^\prime \phi(F), \hspace{2em} \phi(F) = (p^{-1}-1)^{e(F)}\nu^{r(F)}\frac{X_F(H)}{X_F(K_n)},\hspace{2em}
\nu=\left(1+\frac1{\sqrt n}\right)\exp\left(\frac{7\Delta^2}{2\sqrt np}\right),
\end{align}
be our version of $S'_H$ (cf.~the math display above Lemma 4.3 in \cite{riordan2000}, where $\nu$ is replaced by $2$).
Note that $\nu=1+o(1)$. Similarly to~\cite{riordan2000}, we show the following bound.
\begin{lemma}\label{lem:lemma4.3}
We have $S_H\le(1+o(1))\exp\left(\frac{1-p}{p}\alpha^2 N\right)S'_H$.
\end{lemma}
The proof mainly follows and slightly improves upon the corresponding result in~\cite{riordan2000}. We give the details in the appendix (Section~\ref{sec:lemma4.3_proof}). Applying Lemma~\ref{lem:lemma4.2} and Lemma~\ref{lem:lemma4.3} to upper bound $f$, it only remains to show that $S'_H\le 1+o(1)$.

We now start to significantly deviate from the proof in~\cite{riordan2000}, since we treat sparse and dense subgraphs $F \subset H$ differently.
For sparse good subgraphs, we use (a slightly enhanced version of) Lemma 4.4 in \cite{riordan2000}, the proof of which can be found in Section~\ref{sec:lemma4.4_proof}.
\begin{lemma}\label{lem:lemma4.4}
For good spanning subgraphs $F\subset H$, we have
$$\frac{X_F(H)}{X_F(K_n)}\le\left(\frac{3ke}{n}\right)^{r(F)}.$$
\end{lemma}
On the other hand, for dense good subgraphs, we obtain a much stronger bound on $\frac{X_F(H)}{X_F(K_n)}$. This improvement over the bound in~\cite{riordan2000} is specific to powers of Hamilton cycles, and it is vital to our proof.

\begin{lemma}\label{lem:lemma4.4good}
Let $a>0$ be a sufficiently large constant. 
If  $F\subset H$ is a $\zeta'$-dense good spanning subgraph for some $\zeta' = o(1)$ with $\zeta' = \omega(1/\ln{n})$, we have
$$\frac{X_F(H)}{X_F(K_n)}\le\left(\frac{ea^{\sqrt{\zeta'}}}{(1-a\sqrt{\zeta'})n}\right)^{r(F)}.$$
\end{lemma}
We prove Lemma~\ref{lem:lemma4.4good} in Section~\ref{sec:lemma4.4good_proof}.
There are two crucial properties of dense subgraphs that facilitate the improvement in Lemma~\ref{lem:lemma4.4good} compared to Lemma~\ref{lem:lemma4.4}. For one, since the subgraph is dense, most vertices must have full degree $2k$ which in turn significantly restricts the number of possible embeddings into $H$. Additionally, since small subgraphs (that is, components on few vertices) are sparse, dense subgraphs cannot have too many small components and thus not too many components overall.

In order to motivate what follows, we now give an overview of the remainder of the proof. The key idea is the following: Depending on the density of $F$, we apply Lemma~\ref{lem:lemma4.4} or~\ref{lem:lemma4.4good} in order to bound $X_F(H)/X_F(K_n)$, which in turn provides the required bound on $\phi(F)$ and consequently also on $S_H'$ as defined in \eqref{eq:S'def}. 
At first glance, one might think that with this in mind, following~\cite{riordan2000} would easily give our result. However, dealing with $S_H'$ is non-trivial and presents some serious challenges. In \cite{riordan2000}, this is achieved using a bound similar to Lemma~\ref{lem:lemma4.4}, which can be factorized over the components of $F$ (as both $r(F)$ and $e(F)$ are additive over them) and thus every component can be considered separately. Then the partial sums over all connected subgraphs on the same vertex set are evaluated, before establishing an upper bound on $S_H'$.

In our case, there is no simple way to create a factorization over the components, where every component contributes in the same manner to every subgraph it is contained in. We require different upper bounds on $X_F(H)/X_F(K_n)$ depending on the density of $F$, and since the same component can appear in subgraphs with a wide range of densities, any such factorization should depend on the whole of $F$, not just an individual component. One might think that this issue could be circumvented by differentiating based on the density of each component instead of the density of the entire subgraph. However, with this approach, one loses control over the number of components of a subgraph, which is crucial for the argument.

In order to overcome both challenges --- factorizing the sum and bounding the number of components --- we change the order of the factorization step and the step that groups the subgraphs by their vertex set. In \cite{riordan2000} this reversal would correspond to taking the partial sums over all subgraphs on the same vertex set with the same component structure, that is, graphs with the same completion, as defined in Section \ref{sec:notions_notation}.

Our next goal is to determine upper bounds on the partial sums, which can be factorized, considering sparse and dense completions separately. If the completion is sparse, then we can use Lemma~\ref{lem:lemma4.4} as the upper bound for $X_F(H)/X_F(K_n)$ in every summand. On the other hand, for a dense completion we establish an upper bound which roughly corresponds to using Lemma~\ref{lem:lemma4.4good} as an upper bound on $X_F(H)/X_F(K_n)$ in every summand.
However, this can only be achieved as even the weak bound Lemma~\ref{lem:lemma4.4} for sparse graphs with this dense completion gives an insignificant contribution. This then allows for factorization and the proof is completed by summing separately over the sparse and the dense completions.

We now return to the technical details of the proof. As mentioned above, we split $S'_H$ into more than two parts, based not only on the density of the individual $F$, but also on the density of their \emph{completions} $F^*$. More precisely, we split the sum $S'_H$ into
\begin{align}\label{eq:sprimesplit}
S'_H=1+S_{\mathrm s}+S_{\mathrm{ds}}+S_{\mathrm{dd}},
\end{align}
where $1$ accounts for the empty graph. For any other good graph $F$, we differentiate based on whether $F$ is completion-sparse and its \emph{completion} $F^*$ is $\zeta$-sparse. The contribution $S_{\mathrm s}$ is the sum of $\phi(F)$ over all good subgraphs $F$ with $\zeta$-sparse completions $F^*$. The contribution $S_{\mathrm{ds}}$ is the sum of $\phi(F)$ over all completion-sparse good subgraphs $F$ with $\zeta$-dense completions $F^*$, and the last contribution $S_{\mathrm{dd}}$ is the sum of $\phi(F)$ over all completion-dense subgraphs $F$ with $\zeta$-dense completions $F^*$. We will show that the latter three terms are in $o(1)$, employing a different argument in each case. 

In case of $S_{\mathrm{s}}$, we consider graphs with sparse completions. The lower densities allow us to compensate for the weaker bound on $X_F(H)/X_F(K_n)$ provided by Lemma~\ref{lem:lemma4.4}. We prove the following result in Section~\ref{sec:sparse_completions}.
\begin{proposition}\label{prop:sparse_completions}
We have $S_{\mathrm s}=o(1)$.
\end{proposition}

The more interesting case occurs when we are dealing with good graphs that have a dense completion. As mentioned above, this roughly corresponds to using the stronger upper bound on $X_F(H)/X_F(K_n)$ found in Lemma~\ref{lem:lemma4.4good} in every summand. As a first step, we show that this is, in fact, correct. The proof can be found in Section~\ref{sec:dense_completions_sparse}.
\begin{proposition}\label{prop:dense_completions_sparse}
Let $a$ be as in Lemma~\ref{lem:lemma4.4good} and
$$\bar{S}_{\mathrm d} \coloneqq \sum_{\substack{F \in \mathcal{F}^* \\ F \textup{ is } \zeta\textup{-dense}}} p^{-e(F)}\left(\frac{ea^{\sqrt{\zeta'}}}{(1-a\sqrt{\zeta'})n}\right)^{r(F)},$$ 
where $\zeta' = (1+\gamma(n)-\zeta)\zeta$ and $\mathcal{F}^*$ denotes the set of all good completions in $H$.

Then $S_{\mathrm{ds}} + S_{\mathrm{dd}}\le (1+o(1))\bar{S}_{\mathrm{d}}$.
\end{proposition}
This allows us to limit our attention to $\bar{S}_{\mathrm{d}}$, which we prove to be small in Section~\ref{sec:dense_completions_dense_dense}.
\begin{proposition}\label{prop:dense_completions_dense}
We have $\bar{S}_{\mathrm{d}}=o(1)$.
\end{proposition}

We are now ready to give the proof of our main theorem for $G(n,pN)$.
\begin{proof}[Proof of Theorem~\ref{thm:main_uniform}]
Recall that we fixed $p=(1+\varepsilon)(e/n)^{1/k}$. It suffices to prove that $G(n,pN)$ whp contains a copy of $H$. Monotonicity of subgraph containment would then imply the statement of the theorem for $G(n,p'N)$ with $p'>p$. Lemma~\ref{lem:lemma4.2} implies that
$$ f \leq \big(1+o(1)\big) \exp\left(- \frac{1-p}{p}\alpha^2N\right) S_{H}.$$
By Lemma~\ref{lem:lemma4.3},
$$ S_H\le(1+o(1))\exp\left(\frac{1-p}{p}\alpha^2 N\right)S'_H,$$
so $f \leq \left(1+o(1)\right)S'_H$.
Propositions~\ref{prop:sparse_completions} to~\ref{prop:dense_completions_dense} and \eqref{eq:sprimesplit} establish that $S'_H=1+o(1)$, so by~(\ref{eq:secondMoment}) and the fact that $f = \frac{\mathrm{Var}(X)}{\mu^2} + 1$, we have
$$ \mathbb{P}(X=0) \leq f-1 = o(1),$$
where $X$ is the number of copies of $H$ in $G(n,pN)$.
\end{proof}

\subsection{Overview of the rest of the paper}\label{sec:overview}
In Section~\ref{sec:dense_subgraphs}, we discuss the subgraphs $H_s$ (and $P_s$), show that they are maximizers for $\gamma(s)$ and thus obtain a complete description of the map $s\mapsto\gamma(s)$. Importantly, we show that the density of these edge-maximal subgraphs increases with $s$. 
We then use these results to establish properties of dense subgraphs $F\subset H$, specifically that they have few small components, few components in total and many full-degree vertices. Finally, we use these observations to establish Lemma~\ref{lem:lemma4.4good}.

We prove Proposition~\ref{prop:sparse_completions} in Section~\ref{sec:sparse_completions} by rearranging Riordan's proof to facilitate the treatment of disconnected graphs. This allows us to eliminate subgraphs with sparse completions.

Finally, we show Propositions~\ref{prop:dense_completions_sparse} and~\ref{prop:dense_completions_dense} in Section~\ref{sec:dense_completions} by upper bounding the relevant sum in terms of some tail bounds for binomial random variables.

\section{Dense Subgraphs}\label{sec:dense_subgraphs}

Recall Section~\ref{sec:notions_notation}. In this section we analyze dense subgraphs $F\subset H$ in detail. First, in Section~\ref{sec:densest_subgraphs} we show that $H_s$ is the densest subgraph on $s$ vertices, determine $e(H_s)$ and consequently $\gamma(s)$. Then we show that $s\mapsto\gamma(s)$ is increasing and thus $\gamma(n)=\bar{\gamma}(H)$ (cf.~\eqref{eq:gammabar}) is maximal, which thus coincides with $\gamma$ in \cite{riordan2000}.
In Section~\ref{sec:dense_graphs_properties}, we consider good dense subgraphs.
We show that these have few small components, few components in general and that almost all vertices have full degree. In Section~\ref{sec:lemma4.4good_proof} we then use these observations to establish Lemma~\ref{lem:lemma4.4good}. 
\subsection{Densest Subgraphs}\label{sec:densest_subgraphs}
We start by showing that the graphs $H_s$ are edge-maximal.
\begin{lemma}\label{lem:Hs_edg_maximal}
For any $F\subset H$ we have $e(F)\le e(H_{|F|})$.
\end{lemma}
\begin{proof}
Let $F'$ be the subgraph on $V(F)$ induced by $H$, then clearly $e(F)\le e(F')$.
Let $V(F)=\{v_1,\ldots,v_{|F|}\}$, with $1\le v_1<\ldots<v_{|F|}\le n$. We call a set of consecutive vertices in the cyclic order from $V(H)\setminus V(F)$ a gap.
Note that if there is at most one gap, then $F'$ is a copy of $H_{|F|}$.

First assume that there is a gap of length at least $k$, and without loss of generality assume that this includes the vertices $\{n-k+1,\ldots, n\}$.
Then the statement follows by comparing the degrees of $v_i$ in $F'$ and $i$ in $H_{|F|}$, for each $i\in [|F|]$.

Now assume that there exist $j_1<\ldots<j_{k(k-1)/2}$ such that $v_{j_i}\in V(F)$ and $v_{j_i}-1\not\in V(F)$ for $1\le i \le k(k-1)/2$.
Without loss of generality, assume that $j_1>1$. Again, we will compare the degree of $v_i$ in $F'$ and $i$ in $H_{|F|}$. For $1\le i \le k(k-1)/2$, the vertex $j_i$ will have an edge connecting it to $j_i-1$ in $H_{|F|}$, while there is no edge connecting $v_{j_i}$ to $v_{j_i}-1$ in $F'$, thereby $H_{|F|}$ has $k(k-1)/2$ edges that $F'$ does not. On the other hand, $H_{|F|}$ is missing at most $k(k-1)/2$ edges between $\{v_{|F|-k+1},\ldots,v_{|F|}\}$ and $\{v_1,\ldots v_{k}\}$ which are present in $F$.

Finally, if none of the previous assumptions hold, then $|V(H)\setminus V(F)|\le k^3$ and there must exist a set of $k$ consecutive vertices in $V(F)$. Now consider a maximal set of consecutive vertices, with respect to the cyclic order, of size at least $k$.
Roughly speaking, we will show that such a set of vertices can be shifted in one direction until it merges with the next vertex in $V(F)$, without decreasing the number of edges. Without loss of generality, assume that $\{1,\ldots,i\}\subset V(F)$ with $i\ge k$ and $i+1,n \not \in V(F)$.
We may also assume that 
$|\{i+2,\ldots,i+k\}\cap V(F)|\ge |\{n-k+1,\ldots,n-1\}\cap V(F)|$. Define for $1\le j \le |F|$
\[
w_j =
\left\{
	\begin{array}{ll}
		v_j  & \mbox{if } j  \leq i \\
		v_j-1 & \mbox{otherwise}.
	\end{array}
\right.
\]
Denote by $F''$ the graph induced by $\{w_1,\ldots,w_{|F|}\}$ on $H$.
We will show that $e(F'')\ge e(F')$. 
For $v_j\in \{i+2,\ldots,i+k+1\}$, the number of edges between $w_j$ and $\{1,\ldots i\}$ in $F''$ is exactly one more than the number of edges between $v_j$ and $\{1,\ldots i\}$ in $F'$. On the other hand, for $v_j\in \{n-k+1,\ldots,n-1\}$, the number of edges between $w_j$ and $\{1,\ldots,i\}$ in $F''$ is exactly one less than the number of edges between $v_j$ and $\{1,\ldots,i\}$.
By assumption we have $|\{i+2,\ldots,i+k+1\}\cap V(F)|\ge |\{n-k+1,\ldots,n-1\}\cap V(F)|$, and therefore $e(F'')\ge e(F')$. Since we further have $|\{i+2,\ldots,i+k\}\cap V(F'')|\ge |\{n-k+1,\ldots,n-1\}\cap V(F'')|$, we can proceed recursively until $i+1\not\in V(F'')$, thus eliminating the gap between $i$ and $v_{i+1}$. All but one of the remaining gaps can be eliminated in a similar manner, resulting in a copy of 
$H_{|F|}$, which has at least as many edges as $e(F')$, completing the proof.
\end{proof}

This already shows that $H_s$ is the densest connected subgraph on $s$ vertices. 
Next, we determine its number of edges.
\begin{lemma}\label{lem:Hs_edg_cnt}
We have
\begin{align*}
e(P_s) =
\left\{
	\begin{array}{ll}
		\binom{s}{2} & \mbox{if } s  \le k+1 \\
		\binom{s}{2}-\binom{s-k}{2}=ks-\binom{k+1}{2} & \mbox{if } s > k+1,
	\end{array}
\right.
\end{align*}

\begin{align*}
e(H_s) =
\left\{
	\begin{array}{ll}
		e(P_s)  & \mbox{if } s  \le n-k \\
		e(P_s)+\binom{s-(n-k-1)}{2} & \mbox{if } n-k< s \le n.
	\end{array}
\right.
\end{align*}
\end{lemma}
\begin{proof}
For $s\le k+1$, we have $H_s=P_s=K_s$.
For $s>k+1$, start with the clique and then remove downward edges starting with vertex $k+2$ to get
$$e(P_s)=\binom s2-\binom{s-k}2.$$
For $H_s$ with $s>n-k$, we add the edges between the appropriate pairs of vertices $(i,j)$, where $i \in \{n-k+1, \dots, s\}$ and $j \in \{1, \dots, k\}$, yielding $e(H_s)$.
\end{proof}

The next lemma is the primary reason our proof does not work for $k = 3$. Indeed, when $k\leq 3$, we have that $\gamma(H_s)$ is not increasing in $s$.

\begin{lemma}\label{lem:gammas_inc}
The map $\gamma(H_s)$ is increasing for $s>3$ and $\gamma(H_3)=\gamma(H_4)=3$.
\end{lemma}

\begin{proof}
The result follows after a sequence of quick calculations from Lemma~\ref{lem:Hs_edg_cnt}.
\end{proof}

Combining the observations that the $H_s$ are the densest connected subgraphs on $s$ vertices and that their density is increasing in $s$ immediately yields that $H$ is the densest connected subgraph. Recall $F^{\circ}$ from Section~\ref{sec:notions_notation}.

\begin{corollary}\label{cor:densest_subgraphs}
For all $F\subset H$ with $F^\circ\neq\emptyset$ we have $\gamma(F)\le\gamma(H_{|F|})$.
Hence, we have $\gamma(s)=\gamma(H_s)$ for all $s>2$.
\end{corollary}
\begin{proof}
Since $\gamma(F)$ is the expected value of $\gamma(C)$ over $C \in C(F^\circ)$ with respect to $\mu_F$, it is at most the maximum of $\gamma(C)$ over $C(F^\circ)$. 
Lemmas~\ref{lem:Hs_edg_maximal} and~\ref{lem:gammas_inc} complete the proof.
\end{proof}
\subsection{Properties of dense subgraphs}\label{sec:dense_graphs_properties}
In this section, we discuss $\zeta'$-dense subgraphs $F\subset H$ for $\zeta'$ that is $o(1)$ and $\omega(1/\ln(n))$.
In the remainder, we tacitly use the results from Section~\ref{sec:densest_subgraphs}.
Note that $\gamma(n)=k+O(1/n)$ and hence $\gamma(F)\ge\gamma(n)-\zeta'=k-(1-o(1))\zeta'$.
First, we observe that such graphs cannot have many sparse components.
\begin{lemma}\label{lem:sparse_component_cnt} 
Let $F\subset H$ be a $\zeta'$-dense subgraph for some $\zeta' = o(1)$ with $\zeta' = \omega(1/\ln{n})$.
For $\gamma'\le \gamma(n)-\zeta'$ the number of components $C$ of $F^\circ$ with $\gamma(C)\le\gamma'$ is at most
$$\frac{\zeta'|F^\circ|}{\gamma(n)-\gamma'}.$$
\end{lemma}
\begin{proof}
Let $L\subset C(F^\circ)$ be the set of components with $\gamma(C)\le\gamma'$.
Then for $\mu_F(L) = \sum_{C \in L}\mu_F(C)$,
$$\gamma(n)-\zeta'\le\gamma(F)\leq\mu_F(L)\gamma'+(1-\mu_F(L))\gamma(n)$$
and hence $\mu_F(L)\le\zeta'/(\gamma(n)-\gamma')$.
On the other hand, we have $\mu_F(L)\ge |L|/|F^\circ|$.
\end{proof}
Next, we observe that the order of $\zeta'$-dense subgraphs is large and their component count is small.
\begin{lemma}\label{lem:dense_comp_cnt}
For each $\zeta' = o(1)$ with $\zeta' = \omega(1/\ln{n})$, every $\zeta'$-dense subgraph $F\subset H$ has a component $C\in C(F^\circ)$ of order at least $|C|\ge L$ with $L=\min\{s:\gamma(s)\ge\gamma(n)-\zeta'\}=\Theta(1/\zeta')$.
Furthermore, we have
$$c(F^\circ)\le\left(\frac{1}{B}+\sqrt{\zeta'}\right)|F^\circ|,$$
where $B=\min\{s:\gamma(s)>\gamma(n)-\sqrt{\zeta'}\}=\Theta(1/\sqrt{\zeta'})$.
\end{lemma}
\begin{proof}
Since $\gamma(F)$ is an expected value over $\gamma(C)$ for $C\in C(F^\circ)$, there exists a component $C\in C(F^\circ)$ with $\gamma(C)\ge\gamma(n)-\zeta'$.
Lemma~\ref{lem:gammas_inc} and Corollary~\ref{cor:densest_subgraphs} yield
$|C|\ge L$.
By Lemma~\ref{lem:sparse_component_cnt} there are at most $\sqrt{\zeta'}|F^\circ|$ components $C$ with $\gamma(C)\le\gamma(n)-\sqrt{\zeta'}$, so by Lemma~\ref{lem:gammas_inc} the number of components of order less than $B$ is at most $\sqrt{\zeta'}|F^\circ|$.
Furthermore, there can be at most $|F^\circ|/B$ components of order at least $B$.
Finally, $L=\Theta(1/\zeta')$ and $B=\Theta(1/\sqrt{\zeta'})$ follows from Lemma~\ref{lem:Hs_edg_cnt} and Corollary~\ref{cor:densest_subgraphs}.
\end{proof}
Next, we observe that there are not too many ways to embed $\zeta'$-dense subgraphs into $H$.
First, we note that there are not too many ways to embed $H_s$.
\begin{lemma}\label{lem:Hs_aut}
We have $\aut(H_s)=2$ for $2k+1\le s<n$.
\end{lemma}
\begin{proof}
For $2k+1\le s\le n-k$, we have exactly two vertices of degree $d$ for each $k\le d<2k$.
In addition, each vertex with degree $d$ for $k<d<2k$ has exactly one neighbor of degree $k$. Thus, we have two choices where to place one vertex of degree $k$, but this determines the positions of all vertices of degrees less than $2k$. Assume without loss of generality that we chose the identity for these $2k$ vertices. But now, $[k]$ is contained only in the neighborhood of $k+1$, so $k+1$ needs to stay in place. Then $\{2,\dots,k+1\}$ is only contained in the neighbourhood of $k+2$, so this vertex needs to stay in place as well. Iterating this argument completes the proof for this case.

Now let $n-k<s<n$.
Out of the vertices of degree $2k-(n-s)$, exactly two have exactly one neighbor of degree $2k$; namely $1$ and $s$.
Assume without loss of generality that $1$ maps to $1$.
Hence, $k+1$, its only neighbor of degree $2k$, maps to $k+1$.
For $2\le v\le k-(n-s)+1$ there is exactly one vertex of degree $2k-(n-s)$ with exactly $v$ vertices of degree $2k$ as neighbors, one of them being $k+1$, namely vertex $v$; thus $v$ maps to $v$. Furthermore, for $k-(n-s)+2\le v\le k$ there is exactly one vertex of degree $v+k-1$ with $k+1$ as neighbor, namely $v$; so $v$ maps to $v$. This fixes $[k+1]$ and the rest is fixed by their lower $k$ neighbors. 
\end{proof}

In the following we will consider embeddings, i.e.~injections $\iota:V(F)\rightarrow V(G)$ satisfying $\{\iota(u),\iota(v)\}\in E(G)$ for every $\{u,v\}\in E(F)$. The previous lemma implies that once the image of an arbitrary vertex of $H_s$ has been fixed, there are exactly two ways to embed $H_s$ into $H$ for $2k<s<n$.
We use this observation to control the number of embeddings of $\zeta'$-dense $F\subset H$. 

We call a partition $\mathcal{D}(F)$ of $V(F)$ a \emph{consistent partition} if for every $D\in \mathcal{D}(F)$ we have that $F[D]$ forms a copy of $H_s$ for some $s\in\{1,2k+1,\ldots,n\}$.
\begin{lemma}\label{lem:red_vertex_bound}
Let $b$ be a sufficiently large constant. Then for all $n$ large enough, the following holds. 
For any good $\zeta'$-dense subgraph $F\subset H$, where $\zeta' = o(1)$ and $\zeta' = \omega(1/\ln{n})$, there exists a consistent partition $\mathcal{D}(F^{\circ})$ with $|\mathcal{D}(F^{\circ})|\le b \sqrt{\zeta'} r(F^\circ)$.
\end{lemma}

\begin{proof}
Consider the enumeration $v_1<v_2<\ldots<v_{|F^{\circ}|}$ of $V(F^{\circ})=\{v_1,\ldots,v_{|F^{\circ}|}\}$.
Denote by $H'$ the $k$-th power of the Hamilton cycle on $V(F^{\circ})$, maintaining the cyclic order.
Clearly $F^{\circ}\subset H'$ and Lemma~\ref{lem:dense_comp_cnt} implies that the partition $\mathcal{D}(H')$ which consists of one part is consistent.
Repeatedly removing edges from $H'$ will refine the partition, maintaining the invariant that each part forms a $k$-th power of a path in the current graph. 
Removing an edge, where both endpoints are in the same part, will force us to split the part into two disjoint parts, in order to maintain a consistent partition.
Furthermore, if any of these newly created parts have at most $2k$ vertices, then each of the vertices will form their own part.
Therefore, removing any edge subdivides a part into at most $4k$ parts.
Now, the number of edges removed is 
\[
k|F^{\circ}|-e(F^{\circ})\le k|F^{\circ}|-(\gamma(n)-\zeta')(r(F^{\circ})-c(F^{\circ}))\le \zeta'r(F^{\circ})+(k+\gamma(n))c(F^{\circ}),
\]
where we used $|F^{\circ}|=r(F^{\circ})+c(F^{\circ})$ in the last inequality.
By Lemma~\ref{lem:dense_comp_cnt} we have $c(F^\circ)=O(\sqrt{\zeta'})|F^{\circ}|$, which implies $c(F^\circ)=O(\sqrt{\zeta'})r(F^{\circ})$,
completing the proof.
\end{proof}

Denote by $Y_F(G)$ the number of embeddings of $F$ into $G$.

\begin{lemma}\label{lem:dense_emb_count}
Let $F\subset H$ be a graph without isolated vertices and edges, and $\mathcal{D}(F)$ be a consistent partition, then
\[
Y_{F}(H) \le [n]_{c(F)} (4k)^{|\mathcal{D}(F)|}.
\]
\end{lemma}

\begin{proof}
Enumerate the components $C(F^\circ)$ by $C_1,\dots,C_{c(F^\circ)}$.
Fix an arbitrary vertex in each component, which we call the \emph{anchor}.
Now, choose $\iota$ by first choosing the images of the anchors.
Subsequently, choose the images of the remaining vertices in the components $C_i$ for $1\le i \le c(F)$. 
This yields the bound
$$Y_{F}(H) \le [n]_{c(F)} \prod_iZ_i,$$
where $Z_i$ is the number of embeddings of $C_i$ into $H$ that map the anchor to the vertex $1$.
Now, we want to use Lemma~\ref{lem:Hs_aut} to bound $Z_i$. 
Consider the set $\{D\in \mathcal{D}(F): D\subset C_i\}$, and enumerate the elements as $D_1,\ldots,D_{\ell}$, 
such that the anchor is in $D_1$ and the graph induced by $D_1 \cup \ldots \cup D_j$ is connected for every $1\le j \le \ell$.
Using Lemma~\ref{lem:Hs_aut}, we infer that 
the number of embeddings of $H_s$ into $H$ which map a given vertex of $H_s$ to the vertex 1 of $H$ is 1 if $s=1$ and 2 if $s\ge 2k+1$.

Since $|D_1|$ is either 1 or at least $2k+1$, there are at most $4k$ ways to embed it in $H$, such that the anchor maps to 1. 
Now assume that $D_1,\ldots, D_{j-1}$ have been embedded. 
Then there exists a vertex $u\in D_1\cup\ldots \cup D_{j-1}$ and $v\in D_j$ such that $\{u,v\}\in E(C_i)$.
Since $u$ has already been embedded into $H$, there are at most $2k$ choices for the embedding of $v$, and thus at most $4k$ choices for the embedding of $D_j$. \end{proof}

We will need the following related lemma in Section~\ref{sec:dense_completions_dense_dense}

\begin{lemma}\label{lem:num_dense_completions}
There exists a $\eta=(1+o(1))$, such that the number of dense connected subgraphs in $H$ on $s$ vertices is at most $n \eta^s$.
\end{lemma}

\begin{proof}
Let $C$ be a dense connected subgraph of $H$ on $s$ vertices such that $1=v_1<v_2\ldots < v_s$. 
By Lemma~\ref{lem:red_vertex_bound} $V(C)$ has a consistent partition with at most $b\sqrt{\zeta} |C|$ parts. 
Therefore $v_s\le (1+k b\sqrt{\zeta})s$, as the vertices within each part have to be subsequent, and the number of choices for $v_2,\ldots, v_s$ is at most
\[\binom{(1+k b\sqrt{\zeta})s}{s}\le \left(\frac{2e}{kb\sqrt{\zeta}}\right)^{kb\sqrt{\zeta}s}=\eta^s,\]
where $\eta=(1+o(1))$. The result follows as there are $n$ choices to select $v_1$.
\end{proof}

\subsection{Proof of Lemma~\ref{lem:lemma4.4good}}\label{sec:lemma4.4good_proof}
\begin{proof}[Proof of Lemma~\ref{lem:lemma4.4good}]
As in \cite{riordan2000}, we observe that
$$\frac{X_F(H)}{X_F(K_n)}=\frac{1}{n!}Y_F(H).$$
Recall that a $\zeta'$-dense good subgraph $F\subset H$ consists of $F^{\circ}$ and isolated vertices and thus $r(F^{\circ})=r(F)$. 
First embedding $F^{\circ}$ followed by the $i$ many isolated vertices, together with Lemmas~\ref{lem:dense_comp_cnt},~\ref{lem:red_vertex_bound} and~\ref{lem:dense_emb_count}, implies 
$$\frac{X_F(H)}{X_F(K_n)}=\frac{1}{n!}Y_F(H)= \frac{1}{n!} Y_{F^\circ}(H) i!\le \left(\frac{e}{n-c(F^\circ)}\right)^{r(F^\circ)} (4k)^{b\sqrt{\zeta'} r(F^\circ)}\le \left(\frac{ea^{\sqrt{\zeta'}}}{(1-a\sqrt{\zeta'})n}\right)^{r(F)},$$
where $a$ and $b$ are sufficiently large constants, and we used that $(x)_y \geq \left( \frac{x}{e} \right)^y$, for all $0 \leq y \leq x$.
\end{proof}
\section{Sparse completions}\label{sec:sparse_completions}
In this section we establish Proposition~\ref{prop:sparse_completions}, which states that $S_{\mathrm{s}} = o(1)$, where $S_{\mathrm{s}}$ is the contribution to $S_H'$ (defined in \eqref{eq:S'def}) given by graphs with a $\zeta$-sparse completion.
Recall Section~\ref{sec:notions_notation} and
further the decomposition of $S'_H$ from \eqref{eq:sprimesplit}.
Let $\mathcal F^*$ be the set of good completions (i.e.\ the set of all completions formed by the set of good graphs).
For a completion $F \in \mathcal F^*$, let $\mathcal F(F)$ be the set of good subgraphs $F'\subset F$ with completion $F$.
Then we can decompose
\begin{align}\label{eq:sprime_completions}
S'_H=\sum_{F\in\mathcal F^*}S'_H[F],\text{ where }
S'_H[F]\coloneqq \sum_{F'\in\mathcal F(F)}(p^{-1}-1)^{e(F')}\nu^{r(F')}\frac{X_{F'}(H)}{X_{F'}(K_n)}
\end{align}
and similarly obtain the corresponding representation for $S_{\mathrm s}$ by only summing over $S_H'[F]$ for $\zeta$-sparse $F \in \mathcal F^*$.
To bound $S_H'[F]$ for such $F$, the rough bound given by Lemma~\ref{lem:lemma4.4} for sparse good completions $F$ suffices. Furthermore, using that $r(F') = r(F)=r(F^{\circ})$ for all $F'\in \mathcal{F}(F)$, we have
\begin{align*}
S'_H[F]&\le\left(\frac{3\nu ke}{n}\right)^{r(F)}\sum_{F'\in\mathcal F(F)}(p^{-1}-1)^{e(F')}\\
&\le\left(\frac{3\nu ke}{n}\right)^{r(F)} \sum_{m = 0}^{e(F)} \left[\binom{e(F)}{m}(p^{-1}-1)^{m}\right]\\
&=p^{-e(F)}\left(\frac{3\nu ke}{n}\right)^{r(F)}
\le \underbrace{p^{-(\gamma(n)-\zeta)(r(F)-c(F^\circ))}\left(\frac{3\nu ke}{n}\right)^{r(F)}}_{ \eqqcolon \varphi(F)}\,,
\end{align*} where we used the binomial theorem and the fact that $e(F) = \gamma(F)(r(F) - c(F^\circ))$, as well as $\gamma(F) \le \gamma(n) - \zeta$ since $F$ is $\zeta$-sparse. Note that $\varphi(F)$ factorizes if $F$ is the union of two vertex-disjoint graphs $F_1,F_2$ (i.e.\ $\varphi(F_1 \cup F_2) = \varphi(F_1)\varphi(F_2)$).
Extending the summation domain of $S'_H$, we can factorize it and use the Taylor series of the exponential function in a similar way as in \cite{riordan2000} to obtain \begin{align}\label{eq:factorized}
    S_s \le \sum_{t = 1}^\infty \frac{1}{t!} \left( \sum_{\substack{C \in \mathcal{C}(H)\\|C| \ge 3}}\varphi(C) \right)^t = \exp \left( \sum_{\substack{C \in \mathcal{C}(H)\\|C| \ge 3}}\varphi(C) \right) - 1 = \exp\left(
T''  \right) - 1
\end{align} where $\mathcal{C}(H)$ is the set of all connected \emph{induced} subgraphs of $H$ and 
\begin{align*}
    T'' = \sum_{\substack{C \in \mathcal{C}(H)\\|C| \ge 3}}\varphi(C) = \sum_{\substack{C \in \mathcal{C}(H)\\|C| \ge 3}} p^{-(\gamma(n)-\zeta)(|C|-2)}\left(\frac{3\nu ke}{n}\right)^{|C|-1}.
\end{align*} 
Note that the factorization in (\ref{eq:factorized}) is valid since every $F \in \mathcal{F}^*$ can be represented as the disjoint union of connected induced subgraphs of $H$ of order at least $3$, and as we only consider sets of different subgraphs we can compensate for the inherent ordering caused by taking the $t$-th power by $1/t!$.

The number of connected induced subgraphs of $H$ with $s$ many vertices is at most $nk^{s-1}$, as any such subgraph can be constructed by selecting a vertex in $[n]$ and then the remaining vertices in increasing order, with respect to the cyclic order. With the leading $n$ being compensated for by pulling one factor out, we have
\begin{align*}
    T''\le3\nu ek^2\sum_{s\ge 3}\left(\frac{3\nu ek^2}{p^{\gamma(n)-\zeta}n}\right)^{s-2} = 3\nu ek^2\sum_{s\ge 3}\left(\frac{3\nu ek^2}{(1+\varepsilon)^{\gamma(n) -\zeta}n} 
    \left(\frac{n}{e} \right)^{\frac{\gamma(n) - \zeta}{k}} \right)^{s-2} = o(1)
\end{align*} since $\zeta = \omega(1 / \ln n)$ and $\gamma(n)=k+O(1/n)$.
This shows that $S_{\mathrm s} \le \exp(o(1)) - 1 = o(1)$.
\section{Dense completions}\label{sec:dense_completions}
In this section we establish Propositions~\ref{prop:dense_completions_sparse} and~\ref{prop:dense_completions_dense}. We start with the representation of $S'_H$ given by \eqref{eq:sprime_completions} and the corresponding representation of $S_{\mathrm{ds}}+S_{\mathrm{dd}}$ obtained by only summing over dense completions. For every dense completion $F$, we consider completion-sparse and completion-dense $F'$ which have completion $F$ separately. Recall that the property of being completion-sparse (or completion-dense) is defined \emph{relative} to the density of $F$ (cf. Section~\ref{sec:notions_notation}). 

Recall further that $\mathcal{F}(F)$ is the set of all graphs that have completion $F$. In Section~\ref{sec:dense_completions_sparse}, we upper bound $S_H'[F]$ for $\zeta$-dense completions $F$. First, note that a completion-dense subgraph $F'$ of $F$ is itself $\zeta'$-dense for $\zeta' = (1+\gamma(n)-\zeta)\zeta$. Thus, we can apply Lemma~\ref{lem:lemma4.4good} to upper bound $\phi(F')$ for these graphs. In addition, we can apply Lemma~\ref{lem:lemma4.4} to get an upper bound on the sum over the completion-sparse subgraphs of $F$, which turns out to be insignificant in comparison to the upper bound on the sum over the completion-dense subgraphs of $F$. This proves Proposition~\ref{prop:dense_completions_sparse} and allows us to subsequently limit our attention to the sum $\bar{S}_{\mathrm{d}}$. 

Next, we perform the factorization step which is facilitated by $\bar{S}_{\mathrm{d}}$. Following this, we consider $\zeta$-sparse and $\zeta$-dense connected induced subgraphs separately, which is done in Section~\ref{sec:dense_completions_dense_sparse} and Section~\ref{sec:dense_completions_dense_dense}, respectively, thereby completing the proof of Proposition~\ref{prop:dense_completions_dense}. As before, using the low density of sparse subgraphs, we show that their contribution is negligible. On the other hand, the contribution of dense subgraphs is negligible, since there are few ways to embed them into $H$.

\subsection{Sparse subgraphs with dense completions}\label{sec:dense_completions_sparse}

Recall that a non-trivial good subgraph $F'$ with dense completion $F\in\mathcal F^*$ is completion-dense if $e(F')\ge(1-\zeta)e(F)$ or equivalently $\gamma(F')\ge(1-\zeta)\gamma(F)$.
We note that for a $\zeta$-dense $F$, a completion-dense graph $F'$ with completion $F$ is $\zeta'$-dense when setting $\zeta'=(1+\gamma(n)-\zeta)\zeta$. With this, we let $a$ be sufficiently large and define
$$B_{\mathrm s}=\frac{3\nu ke}{n},\quad B_{\mathrm d}=\frac{e\nu a^{\sqrt{\zeta'}}}{(1-a\sqrt{\zeta'})n},\quad S_{\mathrm s}[F] = \left(\frac{B_{\mathrm s}}{B_{\mathrm d}}\right)^{r(F)}\sum_{\substack{F'\in\mathcal F(F) \\ F' \text{completion-sparse}}} p^{e(F)-e(F')}(1-p)^{e(F')},$$
where $B_{\mathrm{s}}$ and $B_{\mathrm{d}}$ are the respective bases of the exponential in Lemma~\ref{lem:lemma4.4} and Lemma~\ref{lem:lemma4.4good}.
Using these lemmas, for any dense $F\subset H$ we have
\begin{align}\label{eq:densesparsesplit}
    S'_H[F] &\le \sum_{\substack{F'\in\mathcal F(F) \\ F' \text{completion-sparse}}}(p^{-1}-1)^{e(F')} B_{\mathrm s}^{r(F')} + \sum_{\substack{F'\in\mathcal F(F) \\ F' \text{completion-dense}}}(p^{-1}-1)^{e(F')} B_{\mathrm d}^{r(F')} \nonumber \\
    &\le B_{\mathrm d}^{r(F)} \left( \left(\frac{B_{\mathrm s}}{B_{\mathrm d}}\right)^{r(F)} \right . \sum_{\substack{F'\in\mathcal F(F) \\ F' \text{completion-sparse}}}(p^{-1}-1)^{e(F')} +\left. \sum_{m=0}^{e(F)} \binom{e(F)}{m} (p^{-1} - 1)^{m} \right) \nonumber \\
    &= p^{-e(F)} B_{\mathrm d}^{r(F)} \left( S_{\mathrm s}[F] + 1 \right),
\end{align}
where we extended the summation domain of the sum over completion-dense $F'$ to all subgraphs of $F$ and then applied the binomial theorem. 

Recalling that for a completion-sparse $F'$, we have $e(F') < (1 - \zeta)e(F)$ and extending the domain of the sum in $S_{\mathrm{s}}[F]$ to all such subgraphs of $F$ yields \begin{align*}
    S_{\mathrm s}[F] &\le \left(\frac{B_{\mathrm s}}{B_{\mathrm d}}\right)^{r(F)}\sum_{m=0}^{(1-\zeta)e(F)} \binom{e(F)}{m}p^{e(F)-m}(1-p)^{m} 
    = \left(\frac{B_{\mathrm s}}{B_{\mathrm d}}\right)^{r(F)} \mathbb{P}( X \geq \zeta e(F)),
\end{align*} 
where 
$X \sim \text{Bin}(e(F), p)$. 
Now, a Chernoff bound as given in Theorem 2.1 in \cite{janson2000}, $p=o(\zeta)$ and $\zeta \ln(\zeta p^{-1})=\omega(1)$ implies
\[
    \mathbb{P}( X \geq \zeta e(F)) \le \exp \left( - \frac{\zeta}{p} \left(\ln\left(\frac{\zeta}{p}\right)-1\right)p e(F)\right)
    = \exp ( - \omega(1) e(F))=\exp(-\omega(1)|F^{\circ}|),
\]
since $e(F) = \frac{1}{2}\sum_{v\in F} \text{deg} (v)\ge \frac{1}{2}|F^\circ|$ because each vertex in $F^\circ$ has degree at least $1$.  
With this and the insight that $r(F) = r(F^\circ) \le |F^\circ|$ for good $F$, we bound \begin{align*}
    S_{\mathrm s}[F] \le \exp(O( r(F)) - \omega(1)|F^\circ|) \le \exp(O( |F^\circ|) - \omega(1)|F^\circ|) = o(1).
\end{align*} 
Using (\ref{eq:densesparsesplit}), this implies  
\begin{align*}
S_{\mathrm{ds}}+S_{\mathrm{dd}} = \sum_{\substack{F \in \mathcal{F}^* \\ F \text{ is } \zeta\text{-dense}}} S_H'[F] \le (1 + o(1)) \bar{S}_{\mathrm d}.
 \end{align*} 
This establishes Proposition~\ref{prop:dense_completions_sparse}. To prove the still missing Proposition~\ref{prop:dense_completions_dense}, we now perform the same factorization into contributions of the connected components as in (\ref{eq:factorized}) to get that $\bar{S}_{\mathrm d}\le\exp(T)-1$, where
$$T\coloneqq \sum_{\substack{C \in \mathcal{C}(H)\\|C| \ge 3}}p^{-e(C)}B_{\mathrm d}^{|C|-1}$$
is the sum over all connected induced subgraphs $C\subset H$ with $|C|\ge 3$.
To show that $T=o(1)$, we consider the split $T=T_{\mathrm s}+T_{\mathrm d}$, where $T_{\mathrm s}$ is the sum over all $\zeta$-sparse $C$, and $T_{\mathrm d}$ accounts for the rest. In the following, we consider $T_{\mathrm s}$ and $T_{\mathrm d}$ separately.

\subsection{Sparse connected induced subgraphs}\label{sec:dense_completions_dense_sparse}

To upper bound the number of $\zeta$-sparse connected induced subgraphs, we again use that every such $C$ on $s$ vertices can be embedded into $H$ in at most $nk^{s-1}$ ways (as used in Section~\ref{sec:sparse_completions}), which yields
$$T_{\mathrm s}\le B_{\mathrm d}\sum_{\substack{C \in \mathcal{C}(H)\\|C| \geq 3}}\left(\frac{B_{\mathrm d}}{p^{\gamma(n)-\zeta}}\right)^{|C|-2}
\le nkB_{\mathrm d}\sum_{s\ge 3}\left(\frac{kB_{\mathrm d}}{p^{\gamma(n)-\zeta}}\right)^{s-2}=o(1).$$
This leaves us with $T=T_{\mathrm d}+o(1)$, so it remains to show that $T_{\mathrm d}=o(1)$.
\subsection{Dense connected induced subgraphs}\label{sec:dense_completions_dense_dense}
Now we consider the $\zeta$-dense connected induced subgraphs.
Using Lemma~\ref{lem:num_dense_completions}, the bound $e(C)\le\gamma(n)(|C|-2)=\left(k + 2k/(n-2)\right)(|C|-2)$, the fact that $B_{\mathrm{d}} = (1+o(1))e/n$, and Lemma \ref{lem:dense_comp_cnt} to conclude that $s = |C| \ge L = \Theta(1/\zeta)$ for all $C$ of interest, we obtain 
\begin{align*}
    T_{\mathrm d} &\le n \sum_{s \ge L} (1+o(1))^{s} p^{-\gamma(n)(s-2)} B_{\mathrm{d}}^{s-1}\\
    &= n (1+o(1)) B_{\mathrm{d}} \sum_{s \ge L} \left( \frac{(1+o(1))B_{\mathrm{d}}}{p^{\gamma(n)}} \right)^{s-2}\\
    & =  e(1+o(1)) \sum_{s \ge L} \left( \frac{ 1+o(1) }{(1 + \varepsilon)^{\gamma(n)}} \left(\frac{n}{e}\right)^{2/(n-2)}\right)^{s-2} = o(1),
\end{align*} 
as desired. Note that we also used $\left(n/e\right)^{2/(n-2)} = 1 +o(1)$. This shows that $T=T_{\mathrm d}+o(1)=o(1)$, so $S_{\mathrm{ds}}+S_{\mathrm{dd}} \le  (1+o(1)) \bar{S}_{\mathrm d}=o(1)$, which concludes the proof of Proposition~\ref{prop:dense_completions_dense}.

\section{Acknowledgements}
This project was initiated during the research workshop in Buchboden, June 2024.
In particular, we want to thank Angelika Steger and Konstantinos Panagiotou, who paved the way for this project.

\appendix
\section{Appendix}
\label{sec:appendix}
For completeness, we provide some technical details that were not included in the main body of the paper. We start with the proof of the lower bound on the threshold for containment of $H$ in $G(n,p)$.

\subsection{Proof of Theorem~\ref{thm:main}}
\label{subsec:proof_of_main_result}
\begin{proof}[Proof of the lower bound in Theorem~\ref{thm:main}]
Let $k\geq 4$ and $p^*=(e/n)^{1/k}$. We have to show that for all $\varepsilon>0$ and $p\le(1-\varepsilon)p^*$, there is no copy of $H$ in $G(n,p)$ whp. We do that via the first moment method. Let $X$ be the random variable counting the copies of $H$ in $G\sim G(n,p)$. Then
$$ \mathbb{E}(X) = X_H(K_n) p^{kn} = \frac{n!}{2n} p^{kn} $$
$$\leq  n! (1-\varepsilon)^{kn} (e/n)^n \leq e^{n \ln{n} - n + O(\ln{n})} e^{- \varepsilon k n} e^n n^{-n} = e^{O(\ln{n}) - \varepsilon k n} = o(1),$$
where in the second inequality we used Stirling's approximation to get an upper bound on $n!$, as well as the bound $1-x \leq e^{-x}$. We have shown $\mathbb{E}(X) = o(1)$, so we can conclude that $\mathbb{P}(X> 0) = o(1)$, as desired.
\end{proof}
\begin{proof}[Proof of the upper bound in Theorem~\ref{thm:main}]
Let $G \sim G(n,p)$ with $p \geq (1+\varepsilon) p^*$ with $p^* = (e/n)^{1/k}$. Note that by an application of the Chernoff bound, whp $G$ has at least $M = (1+ \varepsilon/2) pN$ many edges. We subsample $G'$ by picking $M$ edges of $G$ uniformly at random whenever $e(G) \geq M$ and picking $M$ edges of $K_n \setminus G$ uniformly at random otherwise. By symmetry, the two-step sampling process that yields $G'$ is distributed as $G(n, M)$. We then note that by Theorem~\ref{thm:main_uniform}, $G'$ whp contains a copy of $H$. Since we have that whp $G' \subset G$, it follows that $G$ also whp contains a copy of $H$.
\end{proof}

\subsection{Proof of Lemma~\ref{lem:lemma4.4}}\label{sec:lemma4.4_proof}
\begin{proof}[Proof of Lemma~\ref{lem:lemma4.4}]
We follow the proof of Lemma 4.4 in \cite{riordan2000}. Since $F$ is good, each of its non-trivial components (i.e., components that are not isolated vertices) has size at least $3$.
Thus, we can improve the bound on the number $k'$ of non-trivial components from $k'\le n/2$ to $k'\le n/3$. Plugging in $\Delta=2k$ yields the desired result.
\end{proof}

\subsection{upper bounding $S_H$ in terms of $S'_H$}\label{sec:lemma4.3_proof}
\begin{proof}[Proof of Lemma~\ref{lem:lemma4.3}]
We follow the proof of Lemma 4.3 in~\cite{riordan2000}. Since we make very few changes to it, we only note the points that differ here. In~\cite{riordan2000} and in what follows, the notation $f \lesssim g$ means $f \leq \left(1+o(1)\right)g$, and $e_H(v) = \max \{e(F) : F \subset H, |F| = v\}$ for any integer $v$.

The changes in the proof are due to our new definition of $S'_H$.
Firstly, in (4.4) and also in the first math display on page 135, we change $2$ to $\nu$. Next, instead of the second display on page 135 in \cite{riordan2000}, we have
$$\frac{S[F]}{S'[F]}\le\exp\left(-\frac{7 r(F)\Delta^2}{2\sqrt np}\right)
\sum_{t=0}^\infty(p^{-1}-1)^t(1+n^{-1/2})^t\frac{\beta_v\beta_{v-2}\cdots\beta_{v-2t+2}}{t!},$$
where $v$ is the number of isolated vertices in $F$.
The argument that we can ignore terms with $t > \sqrt{n}$ goes through as in~\cite{riordan2000}, so in place of equation (4.5), we obtain
$$ \frac{S[F]}{S'[F]} \lesssim\exp\left(-\frac{7 r(F)\Delta^2}{2\sqrt np}\right)
\sum_{t=0}^{\sqrt{n}}(p^{-1}-1)^t(1+n^{-1/2})^t\frac{\beta_v\beta_{v-2}\cdots\beta_{v-2t+2}}{t!}.$$

Similarly to~\cite{riordan2000}, we consider two cases. If $v \leq n - \sqrt{n}$, we get
\begin{align}
\label{eq:case1}
\frac{S[F]}{S'[F]}\lesssim\exp\left(-\frac{7 r(F)\Delta^2}{2\sqrt np}+\frac{2\Delta^2}{p}\right).
\end{align}
Since $F$ is good, we have $c(F) \leq v + (n-v)/3$, since each component that is not an isolated vertex has order at least $3$. Thus,
$$ r(F) = n - c(F) \geq n - \left(v + \frac{n-v}{3}\right) = \frac{2}{3} (n-v) \geq \frac{2}{3} \sqrt{n}.$$
Plugging this lower bound on $r(F)$ into~(\ref{eq:case1}), we get
$$\frac{S[F]}{S'[F]}\lesssim\exp\left(-\frac{\Delta^2}{3p}\right) \leq \exp\left(\frac{1-p}{p}\alpha^2 N\right),$$
as desired, completing the first case. In case $v \geq n - \sqrt{n}$, the same proof as in~\cite{riordan2000} goes through.
\end{proof}

\end{document}